\documentclass[review,3p,11pt]{elsarticle}

\usepackage{color}
\usepackage{mathtools}
\usepackage{hyperref} 
\usepackage{setspace}
\usepackage{amsfonts}
\usepackage{amssymb}
\usepackage{amsthm}
\usepackage{bm}
\usepackage{ulem}

\biboptions{numbers,compress}

\numberwithin{equation}{section}

\newtheorem{example}{Example}[section]

\newtheorem{theorem}[example]{Theorem}

 \newtheorem{corollary}[example]{Corollary}
\newtheorem{remark}[example]{Remark}
\newtheorem*{maintheorem*}{Main Theorem}
\allowdisplaybreaks
\numberwithin{equation}{section}

\renewcommand{\i}{\ifmmode\mathit{\mathchar"7010 }\else\char"10 \fi}
\renewcommand{\j}{\ifmmode\mathit{\mathchar"7011 }\else\char"11 \fi}
\newcommand{\R}{\mathbb{R}}

\def\begi{\begin{itemize}}
\def\endi{\end{itemize}}
\def\bega{\begin{array}}
\def\enda{\end{array}}

\def\u{\mathbf{u}}

\def\n{\mathbf{n}}

\def\x{\mathbf{x}}
\def\f{\mathbf{f}}
\def\y{\mathbf{y}}
\def\R{\mathbb{R}}
\def\Xi{{\bm\xi}}

\newenvironment{Assumptions}
{%

\begin{enumerate}}%
{\end{enumerate}}

{%

\begin{enumerate}}%
{\end{enumerate}}

\journal{}

\begin{document}

\title{Dispersive effects in two- and three-dimensional peridynamics}

\author[dei]{A. Coclite\corref{cor}}
\ead{alessandro.coclite@poliba.it}

\author[dmmm]{G. M. Coclite}
\ead{giuseppemaria.coclite@poliba.it}

\author[ial]{G. Fanizza}
\ead{gfanizza@fc.ul.pt}

\author[dmmm]{F. Maddalena}
\ead{francesco.maddalena@poliba.it}

\cortext[cor]{Corresponding author}

\address[dei]{Dipartimento di Ingegneria Elettrica e dell'Informazione (DEI), Politecnico di Bari,\\ Via Re David 200 -- 70125 Bari, Italy}

\address[dmmm]{Dipartimento di Meccanica, Matematica e Management, Politecnico di Bari (DMMM),\\ Via Re David 200 -- 70125 Bari, Italy}
 
\address[ial]{Instituto de Astrofis\'ica e Ci\^encias do Espa\c{c}o, Faculdade de Ci\^encias, Universidade de Lisboa,\\
Edificio C8, Campo Grande, P-1749-016, Lisbon, Portugal}

\begin{abstract}
In this paper we study the dispersive properties related to a model of peridynamic evolution, governed by a non local  initial value problem, in the cases of two and three spatial dimensions.  The features  of the wave propagation characterized by the nontrivial interactions between nonlocality and the regimes of low and high frequencies are studied and suitable numerical investigations   
are exposed.

\end{abstract}

\maketitle

\section*{Introduction}
As it is well known, an  evolution partial differential equation is dispersive if, when no boundary conditions are imposed, its wave solutions spread out in space as they evolve in time. This phenomenon is ruled by the so called {\it dispersion relation} which states that plane waves travel at different speeds according to their wave number, hence the spreading of a wave packet occurs as the time evolves \cite{Ablo, SSb, T}. The dispersion relation is intrinsically  connected to the differential operator characterizing the evolution equation, then when the evolution problem is governed by nonlocal operators, as in the case of peridynamic theory, it is a natural question to investigate the dispersive properties exhibited in these situations. Peridynamics, initiated by S.A. Silling (see \cite{Sill4, Sill, Sill1, Sill3, Sill2}), constitutes a nonlocal mechanical theory in which the evolution problem is governed by an integral operator which takes into account a scale length characterizing the dynamical  interactions inside the material.
In some previous works (\cite{Coclite_2021, CDFMV})  we have focused  on the analysis of the dispersive properties of the nonlocal evolution equation related to a model of peridynamics, studied in detail in \cite{Coclite_2018, EP1, EP, EEM}, by limiting the analysis to the scalar case modeling a one-dimensional infinite material body. In that case the study of the  interplay between nonlocality and dispersion has revealed interesting features related to the asymptotic at low and high frequencies, suggesting new pictures in the framework of wave propagation in continua where nonlocal characteristics are  taken into account. The numerical aspects of the problem have been studied in 
\cite{CCMP, CFLMP, DCFP}.
 
In this paper we extend the analysis of these phenomena to the vectorial case, namely we study the dispersive properties of the peridynamic evolution equation in the physical relevant cases of two and three spatial dimensions. The results obtained show that the two- and three-dimensional cases share the same scaling with the uni-dimensional solutions since the dispersive relations exhibit  the same dependence on the nonlocal interaction length $\delta$ and this is due to the scaling of the elastic parameter $\kappa$ vs the interaction length  $\delta$, hence  this dependence has  universal character,  regardless of the dimension of the system. Moreover, the dispersive relations  depends solely on the modulus  of the  Fourier  variable $\xi$ and the two scenario related to $\xi\delta\ll 1$ and $\xi\delta\gg 1$ are analyzed in detail.

The paper is organized as follows. In Section~\ref{linmod} the general equations governing the peridynamic model studied here are exposed and the results pertaining the linear scalar problem are recalled. In Section~\ref{2dim} and Section~\ref{3dim} the dispersion relations and the relative asymptotic properties for the two and three dimensional problems are obtained. In Section~\ref{results} the results deduced through  the previous analysis are carefully discussed and suitable numerical investigations are deepen by exploiting the evolution of a class of initial data.

\section{The peridynamic model}
\label{linmod}

In~\cite{Coclite_2018} a general model for nonlocal continuum mechanics was proposed and studied by exploiting the analytic aspects of global solutions in energy space in the framework of nonlinear hyperelastic constitutive assumptions. Specifically, the momentum balance  equation of motion of an infinite material body takes the form of the following  initial-value problem:

\begin{equation}
\label{eq:CP}
\begin{cases}
\partial_{tt} \u(t,\x)=(K\u(t,\cdot))(\x),&\quad \x\in \R^N,\: t >0,\\
\u(0,\x)=\u_0(\x),\:\partial_t \u(0,\x)=\mathbf{v}_0(\x),&\quad \x\in\R^N,
\end{cases}
\end{equation}
where
\begin{equation}
\label{eq:operator}
{(K\u(t,\cdot))(\x):= \int_{B_\delta(\x)}\mathbf{f}(\x'-\x,\u(\x')-\u(\x))\,d\x',\quad {\mbox{ for every }}\,\x\in\mathbb{R}^N} \, .
\end{equation}
$N=1,2, 3$ corresponding to physically meaningful cases; $\u\in\R^{N}$ denotes the displacement field and $\delta>0$ characterizes  the non local interaction range governed by the integral kernel~$K$. The internal force $\mathbf{f}:\Omega \rightarrow \mathbb{R}^N$, with $\Omega:=(\mathbb{R}^{N}\setminus\{\mathbf 0\})\times\mathbb{R}^N$, is supposed to satisfy the following general constitutive assumptions:
\begin{Assumptions}
\item  $\f\in C^1(\Omega;\mathbb{R}^N)$;
\item \label{ass:f.2} $\f(-\y,-\u)=-\f(\y,\u),$ for every $(\y,\,\u)\in\Omega\times \R^N$;
\item \label{ass:f.5} there exists a function $\Phi\in C^2(\Omega)$ such that
\begin{equation*}
 \f=\nabla_{\u}\Phi,\qquad \Phi(\y,\u)=\kappa \frac{|\u|^p}{|\y|^{N+\alpha p}}+
 \Psi(\y,\u),\quad {\mbox{ for every }}\, (\y,\,\u)\in\Omega,
\end{equation*}
where $\kappa,\, p,\,\alpha$ are constants such that
\begin{equation*}
\kappa>0,\qquad 0<\alpha<1,\qquad   p\ge 2,
\end{equation*}
and 
\begin{align*}
& \Psi(\y,\mathbf 0)=0\le \Psi(\y,\u),\\&
|\nabla _\u \Psi(\y,\u)|, |D^2_{\u}\Psi(\y,\u)|\le g(\y),\quad {\mbox{ for every }}\,(\y,\,\u)\in\Omega,
\end{align*}
for some non-negative function~$g\in L^2_{\rm{loc}}(\R^N)$.
\end{Assumptions}

Assumption~\ref{ass:f.2} can be seen as a counterpart of Newton's Third Law of Motion (the Action-Reaction Law). Also, assumption~\ref{ass:f.5} states that the material is hyperelastic (the linear elastic case corresponding to~$p=2$ and~$\Psi=0$). In this regard, we will first recall the main results obtained for the dispersive propagation in the uni-dimensional case (see \cite{CDFMV}), where the Cauchy problem \eqref{eq:CP} takes the following form.
\begin{equation}
\begin{cases}
\rho\,u_{tt}(t,x)=-2\,\kappa
\displaystyle\int_{-\delta}^{\delta}\frac{u(t,x)-u(t,x-y)}{|y|^{1+2\,\alpha}}dy
=: \left(Ku(t,\cdot)\right)(x),&\quad t>0,\,x\in\R,\\[10pt]
u(0,x)=v_0(x),&\quad x\in\R,\\[5pt]
u_t(0,x)=v_1(x),&\quad x\in\R\,,
\end{cases}
\label{eq:linear_per}
\end{equation}
where $\delta$, $\kappa$ and $\rho$ are positive real constants and $0<\alpha<1$. Let us  clarify the adopted notation for  vectors: bold quantities, such as $\Xi,\,\x,\,...$ refer to the entire vector, whereas their modulus is simply denoted  by  $\xi,\,x,\,...$ 

In \cite{CDFMV}, the following results are proved.
\begin{theorem}
\label{linsol}
Let $v_0,\,v_1\in\mathcal{S}(\mathbb{R})$ and $0<\alpha <1$.
Then problem  \eqref{eq:linear_per} has the unique solution  
$u:\mathbb{R}_+\times\mathbb{R}\rightarrow \mathbb{R}$   given by   
\begin{equation}
u(t,x)
=\int_{\mathbb{R}} e^{-i\xi x}\left[\widehat{v_0}(\xi)\cos\left(\omega(\xi)\,t\right)
+\frac{\widehat{v_1}(\xi)}{\omega(\xi)}\sin\left(\omega(\xi)\,t\right)\right] d\xi\,,
\label{eq:sol}
\end{equation}
where~$\widehat{v_0}(\xi)$ and  $\widehat{v_1}(\xi)$ represent the Fourier transform\footnote{We use here the non-unitary convention that\label{FOUCO}
$$ \widehat{\bf v}(\Xi):=\frac1{(2\pi)^N}\int_{\R^N} {\bf v}({\bf x})\,e^{i\x\cdot\Xi}\,d\x.$$
In this way, the inversion formula reads
$$ {\bf v}(\x)=\int_{\R^N} \widehat {\bf v}(\Xi)\,e^{-i\Xi\cdot\x}\,d\Xi.$$}
of $v_0(x)$ and $v_1(x)$,
and~$\omega: \mathbb{R}\rightarrow \mathbb{R}^+$ is the dispersion relation defined by 
\begin{equation}
\omega(\xi)=\left(\frac{2\kappa}{\rho\,\delta^{2\alpha}}\,\int_{-1}^{1}\frac{1-\cos (\xi\delta z)}{|z|^{1+2\alpha}}dz\right)^{1/2}\,.
\label{eq:disp_rel}
\end{equation}
\end{theorem}

\begin{theorem}\label{ASYLOANDHI}
For $\delta >0$ and $0<\alpha <1$, the function $\xi \mapsto \omega(\xi)$ given by \eqref{eq:disp_rel} satisfies the following relations.

\begin{equation}
\lim_{\xi\rightarrow 0}\xi^{-2}\omega^2(\xi)= \frac{\kappa\,\delta^{2(1-\alpha)}}{(1-\alpha)\,\rho},
\label{LS:XD}
\end{equation}

\begin{equation}
\lim_{\xi\rightarrow \pm\infty}|\xi|^{-2\alpha}\omega^2(\xi)=
\frac{4\kappa
}{\rho}\,\int_{0}^\infty
\frac{1-\cos \tau}{\tau^{1+2\alpha}}d\tau.
\label{IMP}
\end{equation}
\end{theorem}

This paper is devoted to the analysis of the two-dimensional and three-dimensional counterparts of the dispersion relation \eqref{eq:disp_rel}. Analogous results to the uni-dimensional case will be discussed here and the main differences concerning the dispersive behavior emerging from  the increased  dimension  of the system will be analyzed in detail.

\section{The two-dimensional case}
\label{2dim}
\begin{theorem}
\label{linsol2D}
Let ${\bf v}_0,\,{\bf v}_1\in\mathcal{S}(\R^2)$ and $0<\alpha <1$.
Then problem  \eqref{eq:CP} with $N=2$, $p=2$ and $\Psi=0$ has the unique solution  
$\u:\mathbb{R}^+\times\mathbb{R}^2\rightarrow \mathbb{R}^2$   given by
\begin{equation}
\u(t,\x)
=\int_{\R^2} e^{-i\Xi\cdot \x}\left[\widehat{{\bf v}_0}(\Xi)\cos\left(\omega(\xi)\,t\right)
+\frac{\widehat{{\bf v}_1}(\Xi)}{\omega(\xi)}\sin\left(\omega(\xi)\,t\right)\right] d\Xi\,,
\label{eq:sol2D}
\end{equation}
where~$\widehat{{\bf v}_0}(\Xi)$ and  $\widehat{{\bf v}_1}(\Xi)$ represent the Fourier transforms
of ${\bf v}_0(\x)$ and ${\bf v}_1(\x)$,
and~$\omega: \mathbb{R}^+\rightarrow \mathbb{R}^+$ is the dispersion relation defined by
\begin{equation}
\omega(\xi)=\left(\frac{4\pi\kappa}{\rho\,\delta^{2\alpha}}\,\int_0^1 \frac{1-J_0(\xi \delta z)}{z^{1+2\alpha}}dz\right)^{1/2},
\label{eq:disp_rel2D}
\end{equation}
where $J_0(\cdot)$ is the $0$-th order Bessel function of the first kind.
\end{theorem}

\begin{theorem}
\label{displimit2D}
For $\delta >0$ and $0<\alpha <1$, the following asymptotic relations hold true.
\begin{equation}
\lim_{\xi\rightarrow 0}\xi^{-2}\omega^2(\xi)= \frac{\pi\kappa\,\delta^{2(1-\alpha)}}{2(1-\alpha)\,\rho}\label{LS2D:XD}
\end{equation}
and
\begin{equation}
\lim_{\xi\rightarrow \infty}\xi^{-2\alpha}\omega^2(\xi)=
\frac{4\pi\kappa
}{\rho}\,\int_{0}^\infty
\frac{1-J_0(\tau)}{\tau^{1+2\alpha}}d\tau.
\label{IMP2D}
\end{equation}
\end{theorem}

\begin{corollary}[Radial solutions]\label{cor:isotropic2D}
Let ${\bf v}_0,\,{\bf v}_1\in\mathcal{S}(\R^2)$ and $0<\alpha <1$.
Assume ${\bf v}_0$, ${\bf v}_1$ be such that
\begin{equation}
{\bf v}_0(\x)={\bf v_0}\left(\sqrt{x^2_1+x^2_2}\right)
\qquad\text{and}\qquad
{\bf v}_1(\x)={\bf v}_1\left(\sqrt{x^2_1+x^2_2}\right)\,.
\end{equation}
Then, the unique solution of the Cauchy problem \eqref{eq:CP} is
\begin{equation}
\u(t,x)
=2\pi\int_0^\infty \xi\,J_0(\xi x)\left[\widehat{{\bf v}_0}(\xi)\cos\left(\omega(\xi)\,t\right)
+\frac{\widehat{{\bf v}_1}(\xi)}{\omega(\xi)}\sin\left(\omega(\xi)\,t\right)\right] d\xi\,.
\label{eq:isosol2D}
\end{equation}
\end{corollary}

\begin{remark}\label{rem:12D}
{\rm Notice that the relation \eqref{eq:disp_rel2D}  only depends  on the modulus $\xi$ of the Fourier variable $\Xi$ and 
the solution \eqref{eq:isosol2D} is radial since it only depends  on the modulus $x$ of the vector $\x$.}
\end{remark}
As a consequence of Remark~\ref{rem:12D}, the only way to have spatial anisotropic effects in the dynamic evolution is to deal with anisotropic initial conditions. We now proceed in proving  of the aforementioned theorems.

\begin{proof}[Proof of Theorem \ref{linsol2D}]
The existence and uniqueness of the solution for the Cauchy problem \eqref{eq:CP} has been proven in \cite{Coclite_2018}. To prove that this solution is given by Eq.~\eqref{eq:sol2D}, we proceed by applying the Fourier transform to \eqref{eq:CP}. Indeed, we first have that
\begin{equation}
\partial_{tt}\u(t,\x)=\int_{\R^2} e^{-i\Xi\cdot \x}\partial_{tt}\widehat{\bf u}(t,\Xi) d\Xi\,.
\label{eq:td2}
\end{equation}
In the same way,
\begin{align}
(K{\bf u}&(t,\cdot))({\bf x})\equiv\\ \equiv&-2\kappa\int_{B_\delta(0)}
\frac{{\bf u}(t,{\bf x})-{\bf u}(t,{\bf x}-{\bf y})}{|{\bf y}|^{2+2\alpha}}d{\bf y}
\nonumber\\
=&-2\kappa\int_{\R^2} {\bf \widehat u}(t,\Xi)\,e^{-i \Xi\cdot{\bf x}}
\int_{B_\delta(0)}\frac{1-e^{i\Xi\cdot{\bf y}}}{|{\bf y}|^{2+2\alpha}}d{\bf y}d\Xi
\nonumber\\
=&-2\kappa\int_{\R^2} {\bf \widehat u}(t,\Xi)\,e^{-i \Xi\cdot{\bf x}}
\int_0^\delta \int_0^{2\pi} \frac{1-e^{i \xi y\cos\left(\theta-\theta_\xi\right)}}{y^{1+2\alpha}}d\theta d y d\Xi
\nonumber\\
=&-2\kappa\int_{\R^2} {\bf \widehat u}(t,\Xi)\,e^{-i \Xi\cdot{\bf x}}
\int_0^\delta \left[\int_0^{2\pi} \frac{1-\cos(\xi y\cos\left(\theta-\theta_\xi\right))}{y^{1+2\alpha}}d\theta
-i\underbrace{\int_0^{2\pi} \frac{\sin(\xi y\cos\left(\theta-\theta_\xi\right))}{y^{1+2\alpha}}d\theta}_{=0}\right] d y d\Xi
\nonumber\\
=&-4\pi\kappa\int_{\R^2} {\bf \widehat u}(t,\Xi)\,e^{-i \Xi\cdot{\bf x}}
\int_0^\delta \frac{1-J_0(\xi y)}{y^{1+2\alpha}}d yd\Xi
\nonumber\\
=&-4\pi\kappa\delta^{-2\alpha}\int_{\R^2} {\bf \widehat u}(t,\Xi)\,e^{-i \Xi\cdot{\bf x}}
\int_0^1 \frac{1-J_0(\xi \delta z)}{z^{1+2\alpha}}dzd\Xi
\nonumber\\
=&-\rho\int_{\R^2} {\bf \widehat u}(t,\Xi)\,e^{-i \Xi\cdot{\bf x}}
\omega^2(\xi)d\Xi\,.
\label{eq:K2}
\end{align}
In the third line, we expressed the integral over $B_\delta(0)$ in polar coordinates, namely $\Xi=\xi\left( \cos\theta_\xi,\sin\theta_\xi \right)$ and $\y=y\left( \cos\theta,\sin\theta \right)$ with $\xi,y\in \R^+$ and $\theta,\theta_\xi\in[0,2\pi)$, and we have defined $z:=y/\delta$  in the second last line. Moreover, the under-braced integral is null thanks to the symmetry of the integrand under the shift $\theta\rightarrow\theta+\theta_\xi$ and the invariance of the integral for any integration over an entire period of the integrand.

Now, we make use of Eqs.~\eqref{eq:td2} and \eqref{eq:K2} to formulate our two-dimensional Cauchy problem in Fourier space. We obtain
\begin{equation}
\begin{cases}
\partial_{tt}{\bf \widehat u}(t,\Xi)+\omega^2(\xi){\bf \widehat u}(t,\Xi)=0,&\quad \Xi\in \R^2,\: t >0,\\
\widehat{\u} (0,\Xi)=\widehat{\u_0}(\Xi),\:\partial_t \widehat{\u}(0,\Xi)=\widehat{\mathbf{v}}_0(\Xi),&\quad \Xi\in\R^2,
\end{cases}
\end{equation}
which is solved by
\begin{equation}
\widehat{\u}(t,\Xi)=\widehat{{\bf v}_0}(\Xi)\cos(\omega(\xi)t)+\frac{\widehat{{\bf v}_1}(\Xi)}{\omega(\xi)}\sin(\omega(\xi)t)\,,
\end{equation}
hence, the claim  of theorem follows.
\end{proof}

\begin{proof}[Proof of Theorem \ref{displimit2D}]
To prove Eq.~\eqref{LS2D:XD}, let us first recall that the Bessel function of first kind is given by
\begin{equation}
J_0(x)=\sum_{m=0}^\infty\frac{(-1)^m}{m!\,\Gamma(m+1)}\left(\frac{x}{2}\right)^{2m}\,,
\end{equation}
where $\Gamma(x)$ is Euler Gamma function. Hence, we have that the dispersion relation becomes
\begin{align}
\omega^2(\xi)=&-\frac{4\pi\kappa}{\rho\,\delta^{2\alpha}}\,\int_0^1 \frac{1}{z^{1+2\alpha}}\sum_{m=1}^\infty\frac{(-1)^m(\xi \delta)^{2m}}{m!\,\Gamma(m+1)}\left(\frac{z}{2}\right)^{2m}dz
\nonumber\\
=&\xi^2\frac{\pi\kappa \delta^{2(1-\alpha)}}{\rho}\,\int_0^1 z^{1-2\alpha}dz-\frac{4\pi\kappa}{\rho\,\delta^{2\alpha}}\,\int_0^1 \frac{1}{z^{1+2\alpha}}\sum_{m=2}^\infty\frac{(-1)^m(\xi \delta)^{2m}}{m!\,\Gamma(m+1)}\left(\frac{z}{2}\right)^{2m}dz
\nonumber\\
=&\xi^2\frac{\pi\kappa \delta^{2(1-\alpha)}}{2(1-\alpha)\rho}-\frac{4\pi\kappa}{\rho\,\delta^{2\alpha}}\,\sum_{m=2}^\infty\frac{(-1)^m\xi^{2m} \delta^{2m}}{m!\,\Gamma(m+1)}\int_0^1 \frac{z^{2m-1-2\alpha}}{2^{2m}}dz\,.
\end{align}
Therefore,
\begin{align}
\lim_{\xi\rightarrow 0}\xi^{-2}\omega^2(\xi)
=&\frac{\pi\kappa \delta^{2(1-\alpha)}}{2(1-\alpha)\rho}
-\lim_{\xi\rightarrow 0}\frac{4\pi\kappa}{\rho\,\delta^{2\alpha}}\,\sum_{m=2}^\infty\frac{(-1)^m\xi^{2m-2} \delta^{2m}}{m!\,\Gamma(m+1)}\int_0^1 \frac{z^{2m-1-2\alpha}}{2^{2m}}dz\,
\nonumber\\
=&\frac{\pi\kappa \delta^{2(1-\alpha)}}{2(1-\alpha)\rho}\,,
\end{align}
which proves Eq.~\eqref{LS2D:XD}.

To prove Eq.~ \eqref{IMP2D},  firstly we  define $\tau:=\xi\delta z$ such that the dispersive relation becomes
\begin{equation}
\omega^2(\xi)=\frac{4\pi\kappa}{\rho}\xi^{2\alpha}\,\int_0^{\xi \delta} \frac{1-J_0(\tau)}{\tau^{1+2\alpha}}d\tau\,.
\end{equation}
The desired limit then is
\begin{equation}
\lim_{\xi\rightarrow\infty}\xi^{-2\alpha}\omega^2(\xi)=\frac{4\pi\kappa}{\rho}\,\int_0^\infty \frac{1-J_0(\tau)}{\tau^{1+2\alpha}}d\tau\,,
\end{equation}
eventually proving the theorem.
\end{proof}

\begin{proof}[Proof of Corollary \ref{cor:isotropic2D}]
Firstly, we evaluate the Fourier transform of ${\bf v}_0(\x)$
\begin{align}
\widehat{{\bf v}_0}(\Xi)=&\frac{1}{(2\pi)^2}\int_{\R^2}{\bf v}_0\left(\sqrt{x^2_1+x^2_2}\right)e^{i\x\cdot\Xi}d\x
\nonumber\\
=&\frac{1}{(2\pi)^2}\int_0^\infty\int_0^{2\pi}x\,{\bf v}_0\left(x\right)e^{ix\xi\cos\left(\theta-\theta_\xi\right)}d\theta dx
\nonumber\\
=&\frac{1}{(2\pi)^2}\int_0^\infty x\,{\bf v}_0\left(x\right)J_0(\xi x) dx\,,
\end{align}
where we have adopted polar coordinates $\Xi=\xi\left( \cos\theta_\xi,\sin\theta_\xi \right)$ and $\x=x\left( \cos\theta,\sin\theta \right)$. We obtained then that $\widehat{{\bf v}_0}(\Xi)$
is solely function of $\xi$ rather than $\Xi$. The same occurs for $\widehat{{\bf v}_1}(\Xi)$.
Hence, thanks to the Theorem \ref{linsol2D}, the unique solution of the Cauchy problem \eqref{eq:CP} is
\begin{align}
\u(t,\x)
=&\int_{\R^2} e^{-i\Xi\cdot \x}\left[\widehat{{\bf v}_0}\left(\xi\right)\cos\left(\omega(\xi)\,t\right)
+\frac{\widehat{{\bf v}_1}\left(\xi\right)}{\omega(\xi)}\sin\left(\omega(\xi)\,t\right)\right] d\Xi
\nonumber\\
=&\int_0^\infty \int_0^{2\pi}\xi e^{-i\xi x\cos\left( \theta-\theta_\xi \right)}\left[\widehat{{\bf v}_0}\left(\xi\right)\cos\left(\omega(\xi)\,t\right)
+\frac{\widehat{{\bf v}_1}\left(\xi\right)}{\omega(\xi)}\sin\left(\omega(\xi)\,t\right)\right] d\theta_\xi d\xi
\nonumber\\
=&2\pi\int_0^\infty \xi J_0(\xi x)\left[\widehat{{\bf v}_0}\left(\xi\right)\cos\left(\omega(\xi)\,t\right)
+\frac{\widehat{{\bf v}_1}\left(\xi\right)}{\omega(\xi)}\sin\left(\omega(\xi)\,t\right)\right] d\xi\,,
\end{align}
which indeed only depends on the scalar variable $x$.
\end{proof}

\section{The three-dimensional case}
\label{3dim}
\begin{theorem}
\label{linsol3D}
Let ${\bf v}_0,\,{\bf v}_1\in\mathcal{S}(\R^3)$ and $0<\alpha <1$.
Then problem  \eqref{eq:CP} with $N=3$, $p=2$ and $\Psi=0$ has the unique solution  
$\u:\mathbb{R}^+\times\mathbb{R}^3\rightarrow \mathbb{R}^3$   given by
\begin{equation}
\u(t,\x)
=\int_{\R^3} e^{-i\Xi\cdot \x}\left[\widehat{{\bf v}_0}(\Xi)\cos\left(\omega(\xi)\,t\right)
+\frac{\widehat{{\bf v}_1}(\Xi)}{\omega(\xi)}\sin\left(\omega(\xi)\,t\right)\right] d\Xi\,,
\label{eq:sol3D}
\end{equation}
where~$\widehat{{\bf v}_0}(\Xi)$ and  $\widehat{{\bf v}_1}(\Xi)$ represent the Fourier transforms
of ${\bf v}_0(\x)$ and ${\bf v}_1(\x)$,
and~$\omega: \mathbb{R}^+\rightarrow \mathbb{R}^+$ is the dispersion relation defined by
\begin{equation}
\omega(\xi)=\left(\frac{8\pi\kappa}{\rho\,\delta^{2\alpha}}\,\int_0^1 \frac{1-j_0(\xi \delta z)}{z^{1+2\alpha}}dz\right)^{1/2},
\label{eq:disp_rel3D}
\end{equation}
where $j_0(x):=\frac{\sin x}{x}$ is the $0$-th order spherical Bessel function of the first kind.
\end{theorem}

\begin{theorem}
\label{displimit3D}
For $\delta >0$ and $0<\alpha <1$, the following asymptotic relations hold true.
\begin{equation}
\lim_{\xi\rightarrow 0}\xi^{-2}\omega^2(\xi)= \frac{2\pi\kappa\,\delta^{2(1-\alpha)}}{3(1-\alpha)\,\rho}\label{LS3D:XD}
\end{equation}
and
\begin{equation}
\lim_{\xi\rightarrow \infty}\xi^{-2\alpha}\omega^2(\xi)=
\frac{8\pi\kappa
}{\rho}\,\int_{0}^\infty
\frac{1-j_0(\tau)}{\tau^{1+2\alpha}}d\tau.
\label{IMP3D}
\end{equation}
\end{theorem}

\begin{corollary}[Radial solution]\label{cor:isotropic3D}
Let ${\bf v}_0,\,{\bf v}_1\in\mathcal{S}(\R^3)$ and $0<\alpha <1$.
 Assume ${\bf v}_0$, ${\bf v}_1$ be such that

\begin{equation}
{\bf v}_0(\x)={\bf v}_0\left(\sqrt{x^2_1+x^2_2+x^2_3}\right)
\qquad\text{and}\qquad
{\bf v}_1(\x)={\bf v}_1\left(\sqrt{x^2_1+x^2_2+x^2_3}\right)\,.
\end{equation}
Then, the unique solution of the Cauchy problem \eqref{eq:CP} is
\begin{equation}
\u(t,x)
=4\pi\int_0^\infty \xi^2\,j_0(\xi x)\left[\widehat{{\bf v}_0}(\xi)\cos\left(\omega(\xi)\,t\right)
+\frac{\widehat{{\bf v}_1}(\xi)}{\omega(\xi)}\sin\left(\omega(\xi)\,t\right)\right] d\xi\,.
\label{eq:isosol3D}
\end{equation}
\end{corollary}

\begin{remark}\label{rem:13D}
{\rm Analogously to the two-dimensional case, the dispersive relation \eqref{eq:disp_rel3D} only depends  on the modulus $\xi$ of the Fourier vector $\Xi$. Moreover, the
solution \eqref{eq:isosol3D} is radial on the sphere since it only depends  on the modulus $x$ of the vector $\x$.}
\end{remark}

\begin{proof}[Proof of Theorem \ref{linsol3D}] This proof follows the one already presented for the two-dimensional case. To this end, we need to apply the 3D Fourier transform within the operator $K$ 
\begin{align}
(K{\bf u}(t,\cdot))({\bf x})\equiv&-2\int_{B_\delta(0)}
\frac{{\bf \widehat u}(t,{\bf x})-{\bf \widehat u}(t,{\bf x}-{\bf y})}{|{\bf y}|^{3+2\alpha}}d{\bf y}
\nonumber\\
=&-2\int_{\R^3} \,{\bf \widehat u}(t,\Xi)\,e^{-i \Xi\cdot{\bf x}}
\underbrace{\int_{\mathcal{S}^3(\delta)}\frac{1-e^{i\Xi\cdot{\bf y}}}{|{\bf y}|^{3+2\alpha}}d{\bf y}}d\Xi
\nonumber\\
=&-2\int_{\R^3}\,{\bf \widehat u}(t,\Xi)\,e^{-i \Xi\cdot{\bf x}}
\underbrace{\int_0^\delta \frac{4\pi}{y^{1+2\alpha}}\left( 1-\frac{\sin \xi y}{\xi y} \right) dy}d\Xi
\nonumber\\
=&-8\pi\,\delta^{-2\alpha}\int_{\R^3}\,{\bf \widehat u}(t,\Xi)\,e^{-i \Xi\cdot{\bf x}}
\int_0^1 \frac{1-j_0(\xi\delta z)}{z^{1+2\alpha}}dz  d\Xi\,,
\end{align}
where we have defined $z:=z\delta$ in the last line. The equality between the under-braced terms in the second and third line is given as follows. Since the integral over the sphere is invariant for rotation of the coordinate system, we first align the component $y_3$ with the direction of $\Xi$. In this way, the integral can be written in polar coordinates as
\begin{align}
\int_{\mathcal{S}^3(\delta)}\frac{1-e^{i\Xi\cdot{\bf y}}}{|{\bf y}|^{3+2\alpha}}d{\bf y}
=&\int_0^\infty\int_{-1}^1\int_0^{2\pi}\frac{1-e^{i\xi y\nu}}{|{\bf y}|^{3+2\alpha}} y^2 d\phi d\nu d y
\nonumber\\
=&2\pi\int_0^\infty\int_{-1}^1\frac{1-e^{i\xi y\nu}}{y^{1+2\alpha}} d\nu d y\,,
\end{align}
where $\nu:=\frac{\Xi\cdot{\bf y}}{\xi y}$. At this point, the integration over $\nu$ in the last equation can be written as
\begin{equation}
\int_{-1}^1(1-e^{i\xi y\nu}) d\nu
=\int_{-1}^1(1-\cos\xi y\nu) d\nu-i\underbrace{\int_{-1}^1\sin\xi y\nu d\nu}_{=0}
=2\left(1-\frac{\sin\xi y}{\xi y}\right)\,,
\end{equation}
leading then to
\begin{equation}
(K{\bf u}(t,\cdot))({\bf x})=-\rho\int_{\R^3}\,{\bf \widehat u}(t,\Xi)\,e^{-i \Xi\cdot{\bf x}}
\omega^2(\xi) d\Xi\,,
\end{equation}
with
\begin{equation}
\omega^2(\xi)=\frac{8\pi\kappa}{\rho\,\delta^{2\alpha}}\int_0^1 \frac{dz}{z^{1+2\alpha}}\left( 1-\frac{\sin \xi\delta z}{\xi\delta z} \right)\,.
\end{equation}
Hence, the proof follows in the sam way as the one of Theorem \ref{linsol2D}. Indeed, the three-dimensional Cauchy problem in Fourier space reads
\begin{equation}
\begin{cases}
\partial_{tt}{\bf \widehat u}(t,\Xi)+\omega^2(\xi){\bf \widehat u}(t,\Xi)=0,&\quad \Xi\in \R^3,\: t >0,\\
\widehat\u(0,\Xi)=\widehat\u_0(\Xi),\:\partial_t \widehat\u(0,\Xi)=\widehat{\mathbf{v}}_0(\Xi),&\quad \Xi\in\R^3,
\end{cases}
\end{equation}
which is solved by
\begin{equation}
\widehat{\u}(t,\Xi)=\widehat{{\bf v}_0}(\Xi)\cos(\omega(\xi)t)+\frac{\widehat{{\bf v}_1}(\Xi)}{\omega(\xi)}\sin(\omega(\xi)t)\,,
\end{equation}
hence, the proof of the theorem is concluded.
\end{proof}

\begin{proof}[Proof of Theorem \ref{displimit3D}]
To prove Eq.~\eqref{LS3D:XD}, let us expand the Bessel function $j_0(x)$ as
\begin{equation}
j_0(x):=\frac{\sin x}{x}=\sum_{m=0}^\infty\frac{(-1)^m}{(2m+1)!}x^{2m}\,.
\end{equation}
Hence, the dispersion relation becomes
\begin{align}
\omega^2(\xi)=&-\frac{8\pi\kappa}{\rho\,\delta^{2\alpha}}\int_0^1 \frac{dz}{z^{1+2\alpha}}\sum_{m=1}^\infty\frac{(-1)^m}{(2m+1)!}(\xi\delta z)^{2m}
\nonumber\\
=&\xi^2\frac{4\pi\kappa\delta^{2(1-\alpha)}}{3\rho}\int_0^1 z^{1-2\alpha}dz
-\frac{8\pi\kappa}{\rho\,\delta^{2\alpha}}\int_0^1 \frac{dz}{z^{1+2\alpha}}\sum_{m=2}^\infty\frac{(-1)^m}{(2m+1)!}(\xi\delta z)^{2m}
\nonumber\\
=&\xi^2\frac{2\pi\kappa\delta^{2(1-\alpha)}}{3(1-\alpha)\rho}
-\frac{8\pi\kappa}{\rho\,\delta^{2\alpha}}\int_0^1 \frac{dz}{z^{1+2\alpha}}\sum_{m=2}^\infty\frac{(-1)^m}{(2m+1)!}(\xi\delta z)^{2m}\,,
\end{align}
and then
\begin{equation}
\lim_{\xi\rightarrow 0}\xi^{-2}\omega^2(\xi)=\frac{2\pi\kappa\delta^{2(1-\alpha)}}{3(1-\alpha)\rho}\,,
\end{equation}
which proves Eq.~\eqref{LS3D:XD}.

For the second part of the proof, we first define $\tau:=\xi\delta z$ such that the dispersive relation becomes
\begin{equation}
\omega^2(\xi)=\frac{8\pi\kappa}{\rho}\xi^{2\alpha}\,\int_0^{\xi \delta} \frac{1-j_0(\tau)}{\tau^{1+2\alpha}}d\tau\,.
\end{equation}
The desired limit then is
\begin{equation}
\lim_{\xi\rightarrow\infty}\xi^{-2\alpha}\omega^2(\xi)=\frac{8\pi\kappa}{\rho}\,\int_0^\infty \frac{1-j_0(\tau)}{\tau^{1+2\alpha}}d\tau\,,
\end{equation}
eventually proving the theorem.
\end{proof}

\begin{proof}[Proof of Corollary \ref{cor:isotropic3D}]
Just as for the proof of Corollary \ref{cor:isotropic2D}, it is enough to prove that $\widehat{{\bf v}_0}$ and $\widehat{{\bf v}_1}$ only depends  on the modulus $\xi$.
Indeed, we have 
\begin{align}
\widehat{{\bf v}_0}(\Xi)=&\frac{1}{(2\pi)^3}\int_{\R^2}{\bf v}_0\left(\sqrt{x^2_1+x^2_2+x^2_3}\right)e^{i\x\cdot\Xi}d\x
\nonumber\\
=&\frac{1}{(2\pi)^3}\int_{-1}^1\int_0^\infty\int_0^{2\pi}x^2\,{\bf v}_0\left(x\right)e^{ix\xi\nu}d\phi d\nu dx
\nonumber\\
=&\frac{2}{(2\pi)^2}\int_0^\infty x^2\,{\bf v}_0\left(x\right)j_0(\xi x) dx\,,
\end{align}
where we have adopted polar coordinates and aligned the $z$-axis of the reference  system with the direction of $\x$ and defined $\nu:=\frac{\Xi\cdot\x}{\xi x}$. We obtained then that $\widehat{{\bf v}_0}(\Xi)$
is solely function of $\xi$ rather than $\Xi$. The same occurs for $\widehat{{\bf v}_1}(\Xi)$.
\end{proof}

\section{Results and discussion}
\label{results}
\subsection{Analytical findings} A direct comparison with the one-dimensional case shows that the two- and three-dimensional cases share the same scaling with the uni-dimensional solutions. As a matter of fact, the dispersive relations \eqref{eq:disp_rel2D} and \eqref{eq:disp_rel3D} have the same dependence on $\delta^{-2\alpha}$ occurring in \eqref{eq:disp_rel} as $\xi\rightarrow 0$ and $\xi \rightarrow \infty$. Since this dependence is intimately related to the scaling of $\kappa$ vs $\delta$, this feature is a hint of the fact that this dependence is universal regardless of the dimension of the system.

Moreover, the dispersive relation depends solely on the modulus $\xi$ rather than the entire Fourier vector $\Xi$, no matter what the dimension of the system is. This behavior is somehow expected for small value of $\xi$, i.e. $\xi\delta\ll 1$. Indeed, in this regime the spatial scale of the Fourier mode is much greater than the peridynamics radius $\delta$: here the dynamics of the standard elastic theory must be recovered and then the dispersive relation must depend only on $\xi$. On the other hand, having recovered this isotropy also on small scales Fourier modes, where $\xi \delta \gg 1$, is less trivial. We address this interesting feature to the fact that the integral peridynamic kernel $K$ averages with the same weight all the possible directions of the shift vector $\u(\x)-\u(\x-\y)$.
\begin{table}[ht!]
\begin{center}
\begin{tabular}{|c|c|c|}
\hline
 & $\xi\delta\ll 1$ & $\xi\delta\gg 1$\\
\hline
\rule[-4mm]{0mm}{1cm}
$N=1$ & $\frac{\kappa\,\delta^{2(1-\alpha)}}{(1-\alpha)\,\rho}\xi^2$ & $-\frac{4\kappa}{\rho}\cos(\pi\alpha)\Gamma(-2\alpha)\,\xi^{2\alpha}$\\
\hline
\rule[-4mm]{0mm}{1cm}
$N=2$ & $\frac{\pi\kappa\,\delta^{2(1-\alpha)}}{2(1-\alpha)\,\rho}\xi^2$ & $-\frac{2^{1-2\alpha}\pi\kappa}{\rho}\frac{\Gamma(-\alpha)}{\Gamma(1+\alpha)}\xi^{2\alpha}$\\
\hline
\rule[-4mm]{0mm}{1cm}
$N=3$ & $\frac{2\pi\kappa\,\delta^{2(1-\alpha)}}{3(1-\alpha)\,\rho}\xi^2$ & $\frac{8\pi\kappa}{\rho}\cos(\pi\alpha)\Gamma(-1-2\alpha)\,\xi^{2\alpha}$\\
\hline
\end{tabular}
\caption{Dispersive relations in the limit regimes of large ($\xi\delta\ll 1$) and small ($\xi\delta\ll 1$) scales in terms of the dimensionality of the system.}
\label{tab:resume}
\end{center}
\end{table}

To thoroughly characterize the dispersive relation, we resume in Table \ref{tab:resume} the asymptotic limits of $\omega^2(\xi)$ for $\xi\delta\ll 1$ and $\xi\delta\gg 1$ for different values of the dimension $N$. As a further step with respect to the claims of Theorems \ref{IMP}, \ref{IMP2D} and \ref{IMP3D}, the small scales limits have been formally written in terms of the Euler Gamma function $\Gamma(x)$ and these are shown in Fig.\ref{fig:small}.
\begin{figure}[ht!]
\centering
\includegraphics[scale=1]{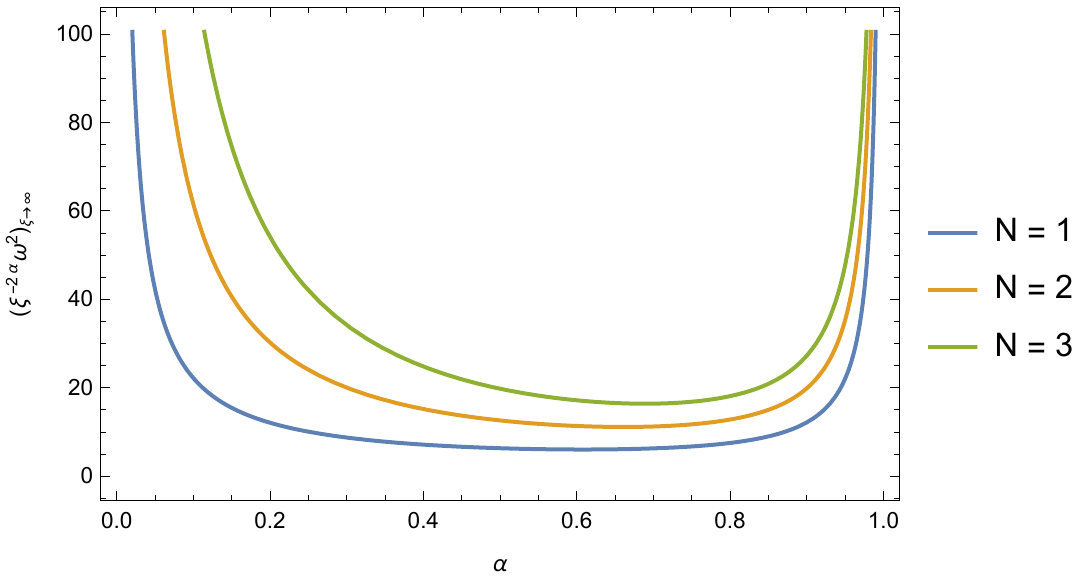}
\caption{Small scales ($\xi\delta\gg 1$) amplitude of the dispersive relation versus the value of $\alpha$ for different dimensions, with $\kappa=\rho=1$.}
\label{fig:small}
\end{figure}
Two different features concerning  small and large scales regimes clearly emerge:
\begin{itemize}
\item for $\xi\delta\ll 1$, the limits for different values of $N$ have different amplitudes. However, they shares an universal scaling with the constitutive parameters of the peridynamics $\delta$ and $\alpha$;
\item on the contrary, when $\xi\delta\gg 1$, the overall amplitudes do not depend on $\delta$. Nevertheless, the dependence on $\alpha$ changes with $N$. From Fig. \ref{fig:small}, it is evident that the qualitative behavior is the same, even though the lower the dimension, the smaller the amplitude. This feature is much evident for small values of $\alpha$.
\end{itemize}

\subsection{Numerical experiments}

To conclude this work, we provide a numerical study for the propagation of elastic deformations in the two and three dimensional case. The analysis that we are going to present is analogous to the one for the one dimensional case in \cite{CDFMV,Coclite_2021}. For the purposes of this paper, we will show numerical evidences that the speed of propagation for an elastic deformation is weakly dependent on the dimensionality of the system, whereas the constitutive micro-scale parameters $\alpha$ and $\delta$ are the main ingredients to estimate such speed of propagation. To this end, we study numerical solutions in correspondence of the following initial conditions
\begin{equation}
v_0(r)=\left( 2\pi \right)^{N/2}\exp\left(-\frac{r^2}{2\,\sigma^2}\right)
\qquad\text{and}\qquad
v_1(r)=0\,,
\label{eq:Cauchy}
\end{equation}
in complete analogy with  the one dimensional system. We also remark that we are limiting our case study by taking $v_0$ and $v_1$ as function in $\R$ rather than $\R^N$. This choice has been done for the sake of simplicity since our conclusion can be immediately extended to the general case.

According to our convention, the Fourier transforms of the initial conditions \eqref{eq:Cauchy} then become
\begin{equation}
\widehat{v_0}(\xi)=\sigma^N\exp\left(-\frac{\xi^2\sigma^2}{2}\right)
\qquad\text{and}\qquad
\widehat{v_1}(\xi)=0\,,
\label{eq:Cauchy_Fourier}
\end{equation}
which satisfy the assumptions of Corollaries \ref{cor:isotropic2D} and \ref{cor:isotropic3D}. The role of $\sigma$ in the initial conditions is crucial to put in evidence the dispersive features due to peridynamic. Indeed, Eq.~\eqref{eq:Cauchy_Fourier} undisputedly privileges the Fourier modes such that $\xi\sigma\lesssim 1$ in the imprinted deformations. The peculiar scales $\xi$ and $\sigma$ then must be compared with the unique constitutive scale $\delta$, since the latter marks the change in the regime for the dispersive relation.

In details, based on the results obtained for the one dimensional case \cite{CDFMV,Coclite_2021}, we expect also for the multi-dimensional scenarios that:
\begin{itemize}
\item[{\it i})]{When $\sigma\gtrsim\delta$, the dispersion due to the peridynamic kernel is hidden since only large scale modes are switched on, given the hierarchy $\xi\lesssim\sigma^{-1}\lesssim\delta^{-1}$. With reference to Table \ref{tab:resume}, these modes can only be dispersed in qualitative agreement with the standard theory of elasticity and the peridynamics parameters may play a role only in the effective group velocity.}
\item[{\it ii})]{Conversely, when $\sigma\ll\delta$, the above-mentioned hierarchy among the scales translates as $\xi\lesssim\sigma^{-1}$ {\it and} $\delta^{-1}\ll\sigma^{-1}$: in principle, both modes $\xi>\delta^{-1}$ and $\xi<\delta^{-1}$ can propagate throughout the system, enabling then the peridynamics evolution to be investigated. Moreover, as a matter of fact, smaller scale modes will travel slower than the large ones, due to the sub-linearity of the dispersive relation.}
\end{itemize}

The understanding described so far is based on the results shown in the previous sections and the detailed study of the one-dimensional case. In the following, we will provide numerical evidences for our claimed phenomenology. To this end, for both the two and three dimensional propagation, we will keep fixed $\kappa=\rho=\delta=1$ and consider the elucidative cases of $\alpha=\{1/10,1/2,9/10\}$ and $\sigma=1$ or $\sigma=1/10$, respectively case {\it i}) and {\it ii}) in the item list above. The choice for $\delta$ to be equal to $1$ might look peculiar. However, it will simplify our discussion concerning the speed of propagation of the elastic deformations, since it will hide all the factors $\delta^{\alpha-1}$ in Table \ref{tab:resume}. We remark that this will not provide any loss of generality if our numerical investigations, since what physically matters to enlighten the dispersive phenomena is the ratio $\delta/\sigma$ rather than their individual amplitudes.

We begin by focusing on the two-dimensional case for the case {\it i}) in Fig.~\ref{fig:N2_s1} and case {\it ii}) in Fig.~\ref{fig:N2_s110}. For $\sigma=1$, Fig.~\ref{fig:N2_s1} exhibits a really mild dispersion as time goes on, regardless on the value of $\alpha$. What changes instead is the speed of propagation of the deformation imprinted by the initial conditions. An estimate of this speed of propagation can be actually obtained from Table \ref{tab:resume}. In fact, for the science case $\sigma=\delta=1$, a non-rigorous but nonetheless accurate estimate for the speed of propagation can be inferred by the group velocity\footnote{A detailed Proof of equivalence between the velocity of the energy flux and the group velocity for to one-dimensional linear peridynamics has been provided in \cite{Coclite_2021}. For the two and three dimensional case, we limit our analysis to the numerical evidences.} $v_g:=\omega'(\xi)$. When $\xi\delta\lesssim1$, we have that $v_g\approx \delta^{1-\alpha}\sqrt{\frac{\pi\kappa\,}{2(1-\alpha)\,\rho}}=\sqrt{\frac{\pi\,}{2(1-\alpha)}}$: the higher $\alpha$, the faster the propagation. In details,
\begin{table}[ht!]
\begin{center}
\begin{tabular}{|c|c|}
\hline
 $\alpha$ & $v_g$\\
\hline
$9/10$ & $3.96$ \\
\hline
$1/2$ & $1.77$\\
\hline
$1/10$ & $1.32$\\
\hline
\end{tabular}
\caption{Estimation of the large scale group velocity for the two-dimensional case with $\kappa=\rho=\delta=1$ for the chosen values of $\alpha$ in Fig.~\ref{fig:N2_s1}.}
\label{tab:N2_group}
\end{center}
\end{table}
despite its roughness, our estimation for the speed of propagation of the signals is in quite good agreement with what emerges from Fig.~\ref{fig:N2_s1}.
\begin{figure}[ht!]
\centering
\includegraphics[scale=0.26]{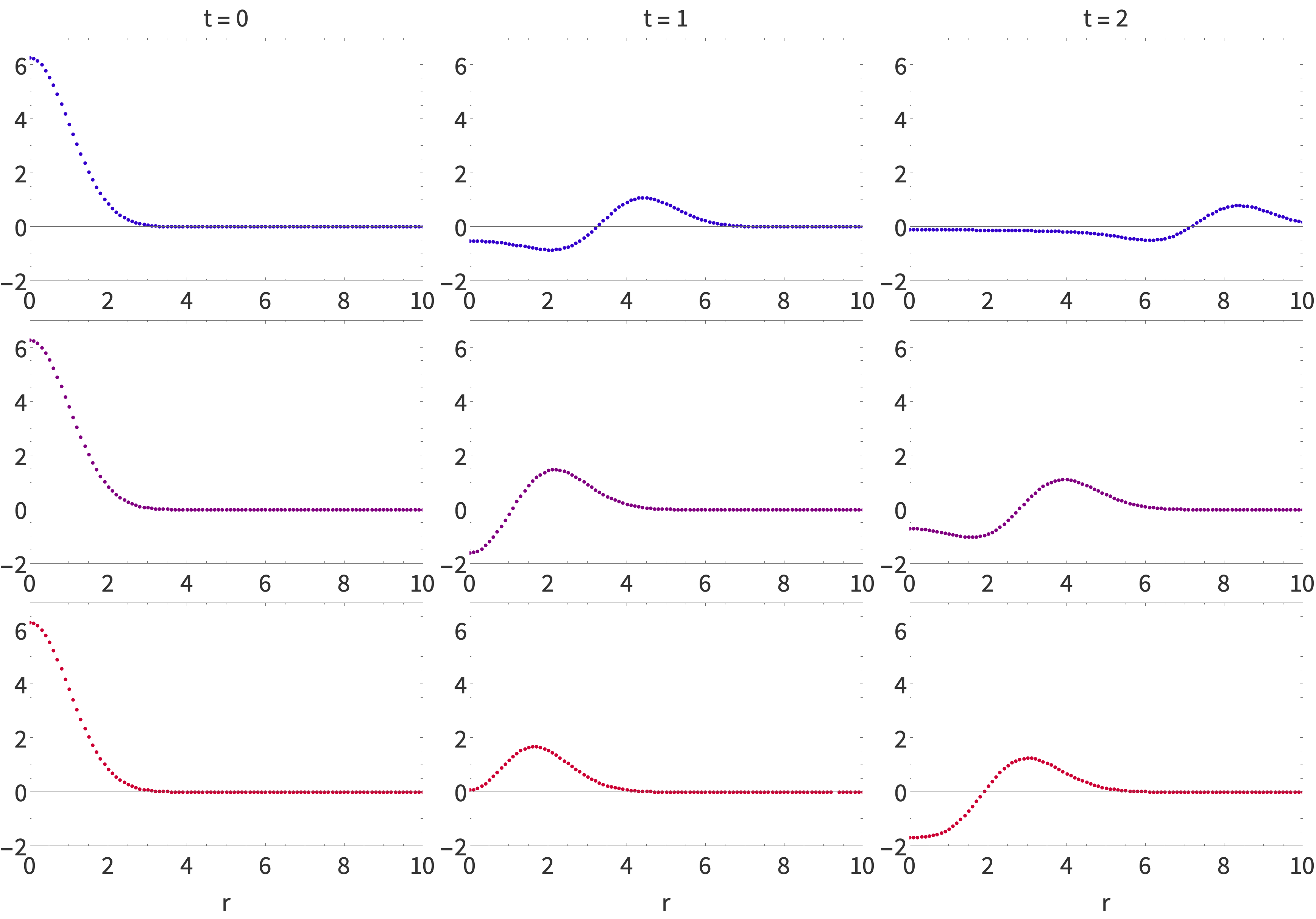}
\caption{Isotropic 2D solution with $\kappa=\rho=\delta=1$ and $\sigma=1$ for $t=0,1,2$ (from right to left panels) and for $\alpha=9/10,1/2,1/10$ (from top to bottom panels).}
\label{fig:N2_s1}
\end{figure}

The same numerical investigation has been repeated for the case {\it ii}) with $\sigma=1/10$ and reported in Fig.~\ref{fig:N2_s110}. Two main differences are evident between the two cases: when $\sigma=1/10$, the dispersion is evident and the looks slower than the case {\it i}).
\begin{figure}[ht!]
\centering
\includegraphics[scale=0.26]{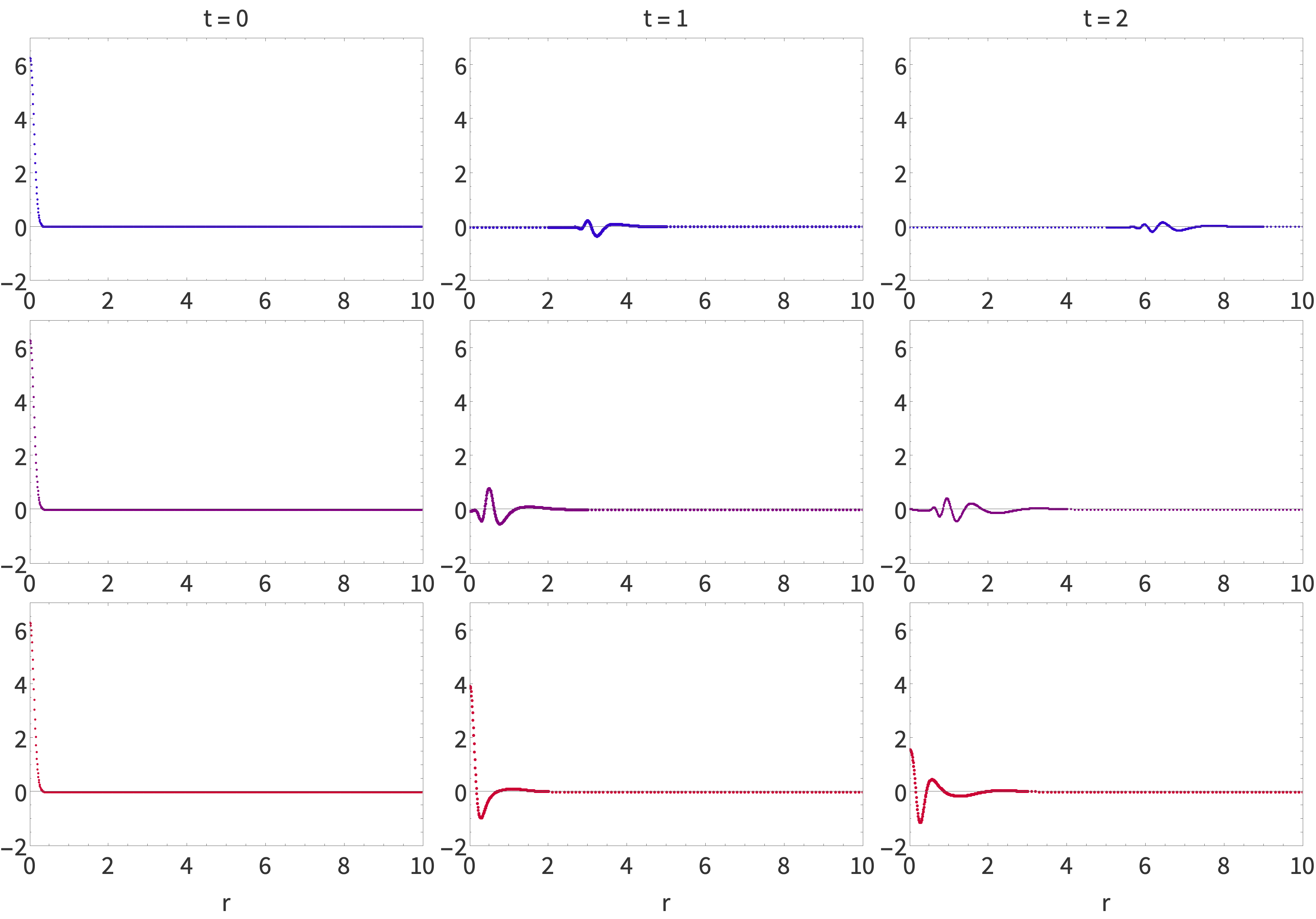}
\caption{Isotropic 2D solution with $\kappa=\rho=\delta=1$ and $\sigma=1/10$ for $t=0,1,2$ (from right to left panels) and for $\alpha=9/10,1/2,1/10$ (from top to bottom panels).}
\label{fig:N2_s110}
\end{figure}
These features can be explained again by looking at the Table \ref{tab:resume}. Indeed, by invoking the role of the velocity group also for the case {\it ii}), all the scales such that $\xi\gg\delta$ travel with group velocity which is suppressed by a factor $\xi^{\alpha-1}$ and this suppression is undisputedly more severe for smaller values of $\alpha$.

The phenomenology of the three dimensional cases is qualitatively the same as the the two dimensional ones and this can be appreciated in Figs.~\ref{fig:N3_s1} and \ref{fig:N3_s110} respectively for $\sigma=1$ and $\sigma=1/10$.
\begin{figure}[ht!]
\centering
\includegraphics[scale=0.26]{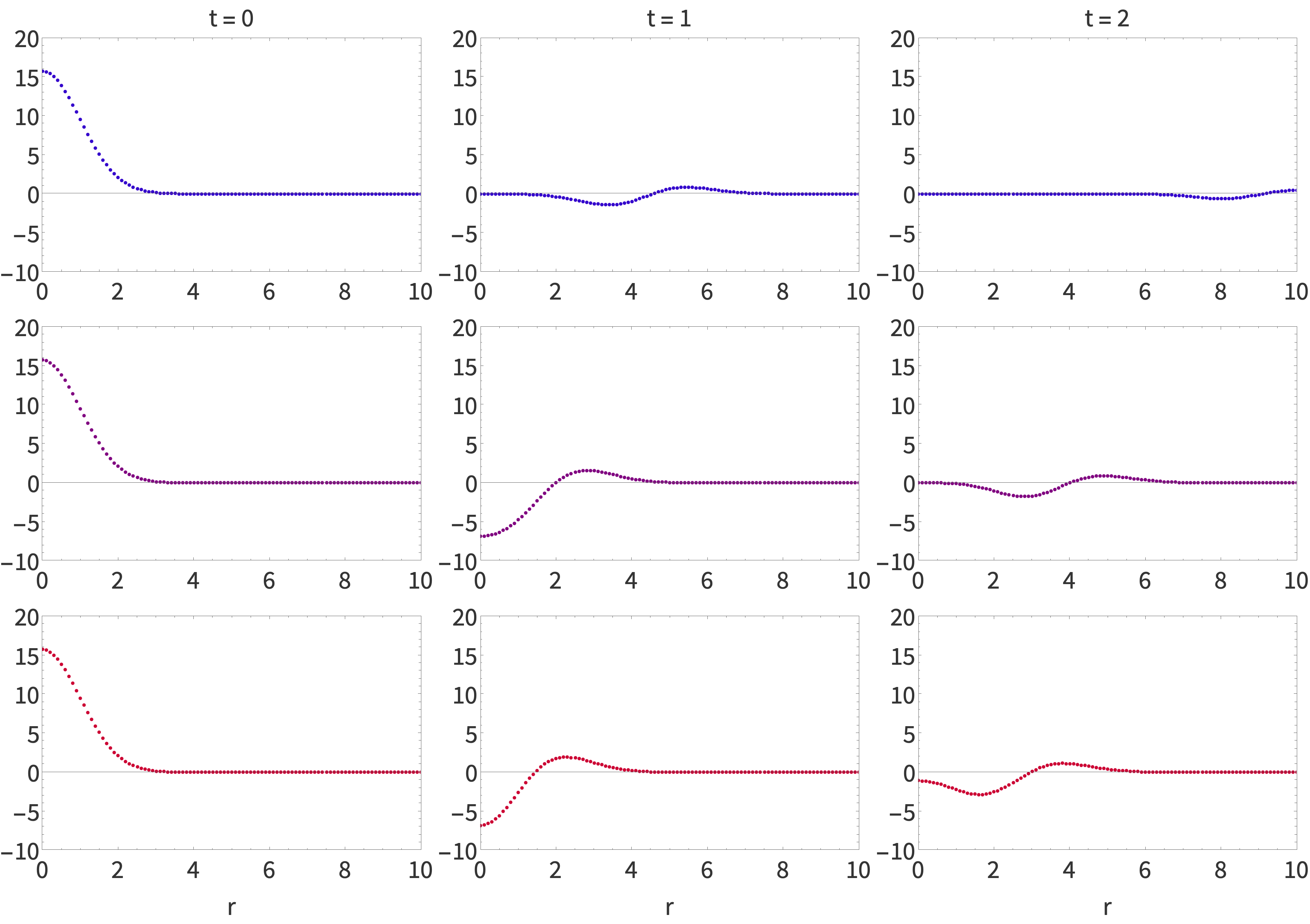}
\caption{Isotropic 3D solution with $\kappa=\rho=\delta=1$ and $\sigma=1$ for $t=0,1,2$ (from right to left panels) and for $\alpha=9/10,1/2,1/10$ (from top to bottom panels).}
\label{fig:N3_s1}
\end{figure}

The main quantitative disclaimer w.r.t. to the two dimensional scenario can be provided in regard of the group velocity for the large scale modes (i.e. $\xi\delta\lesssim 1$). From Table \ref{tab:resume}, we get that the large scale group velocity for the three dimensional case is $v_g\approx \delta^{1-\alpha}\sqrt{\frac{2\pi\kappa}{3(1-\alpha)\,\rho}}=\sqrt{\frac{2\pi}{3(1-\alpha)}}$.
\begin{table}[ht!]
\begin{center}
\begin{tabular}{|c|c|}
\hline
 $\alpha$ & $v_g$\\
\hline
$9/10$ & $4.58$ \\
\hline
$1/2$ & $2.05$\\
\hline
$1/10$ & $1.53$\\
\hline
\end{tabular}
\caption{Estimation of the large scale group velocity for the three-dimensional case with $\kappa=\rho=\delta=1$ for the chosen values of $\alpha$ in Fig.~\ref{fig:N3_s1}.}
\label{tab:N2_1_group}
\end{center}
\end{table}

\begin{figure}[ht!]
\centering
\includegraphics[scale=0.26]{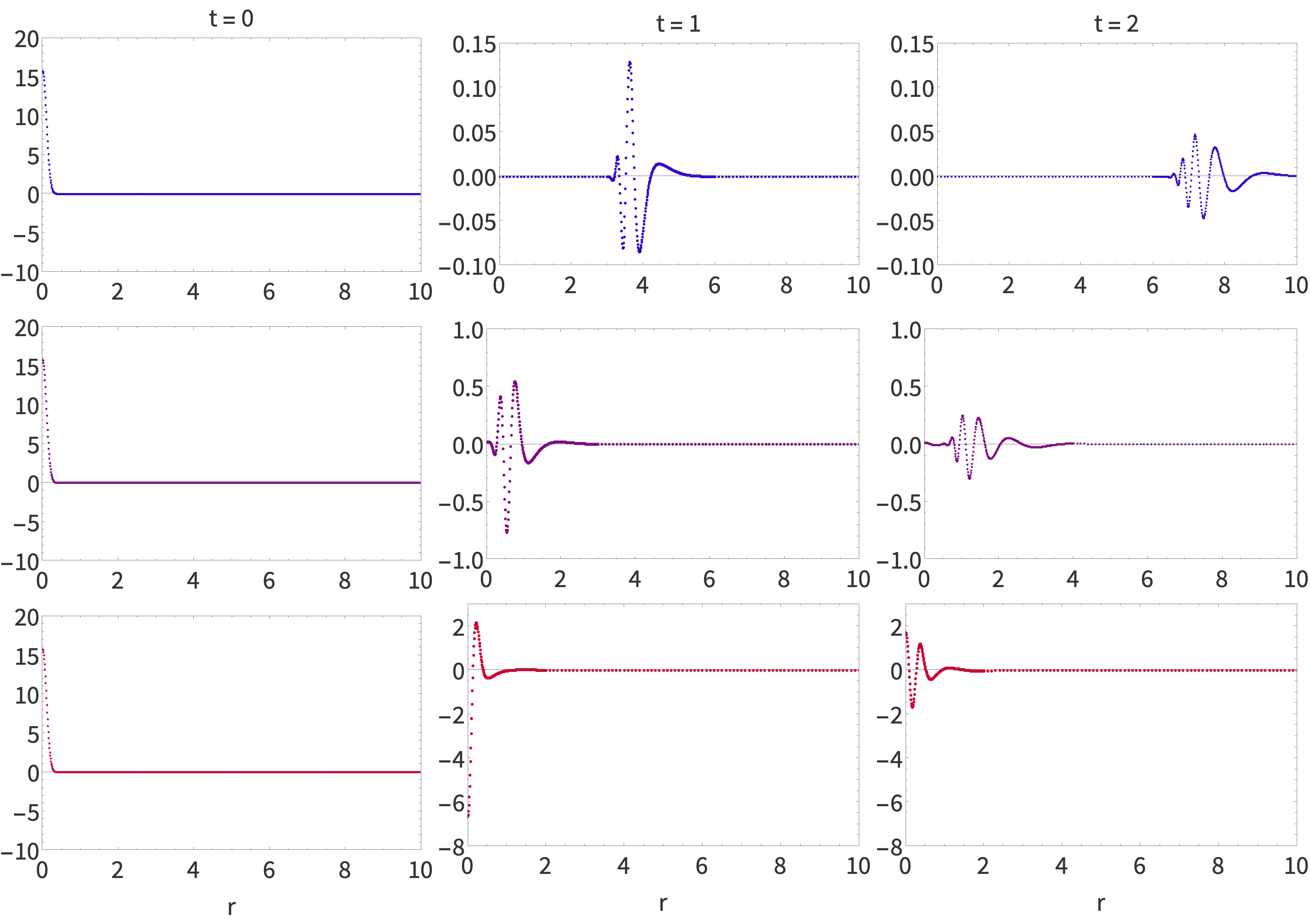}
\caption{Isotropic 3D solution with $\kappa=\rho=\delta=1$ and $\sigma=1/10$ for $t=0,1,2$ (from right to left panels) and for $\alpha=9/10,1/2,1/10$ (from top to bottom panels).}
\label{fig:N3_s110}
\end{figure}
Interestingly, for the three-dimensional case the intrinsic velocity of the deformation is slightly higher than the two dimensional case and this is consistent with Fig.~\ref{fig:N3_s1}. From a more qualitative point of view, this speed of propagation is suppressed for smaller scales perturbations exactly as expected by the $\xi^{\alpha-1}$ suppression already discussed for the two dimensional case.

\section{Anisotropic initial conditions}
\begin{figure}[ht!]
\centering
\includegraphics[scale=0.21]{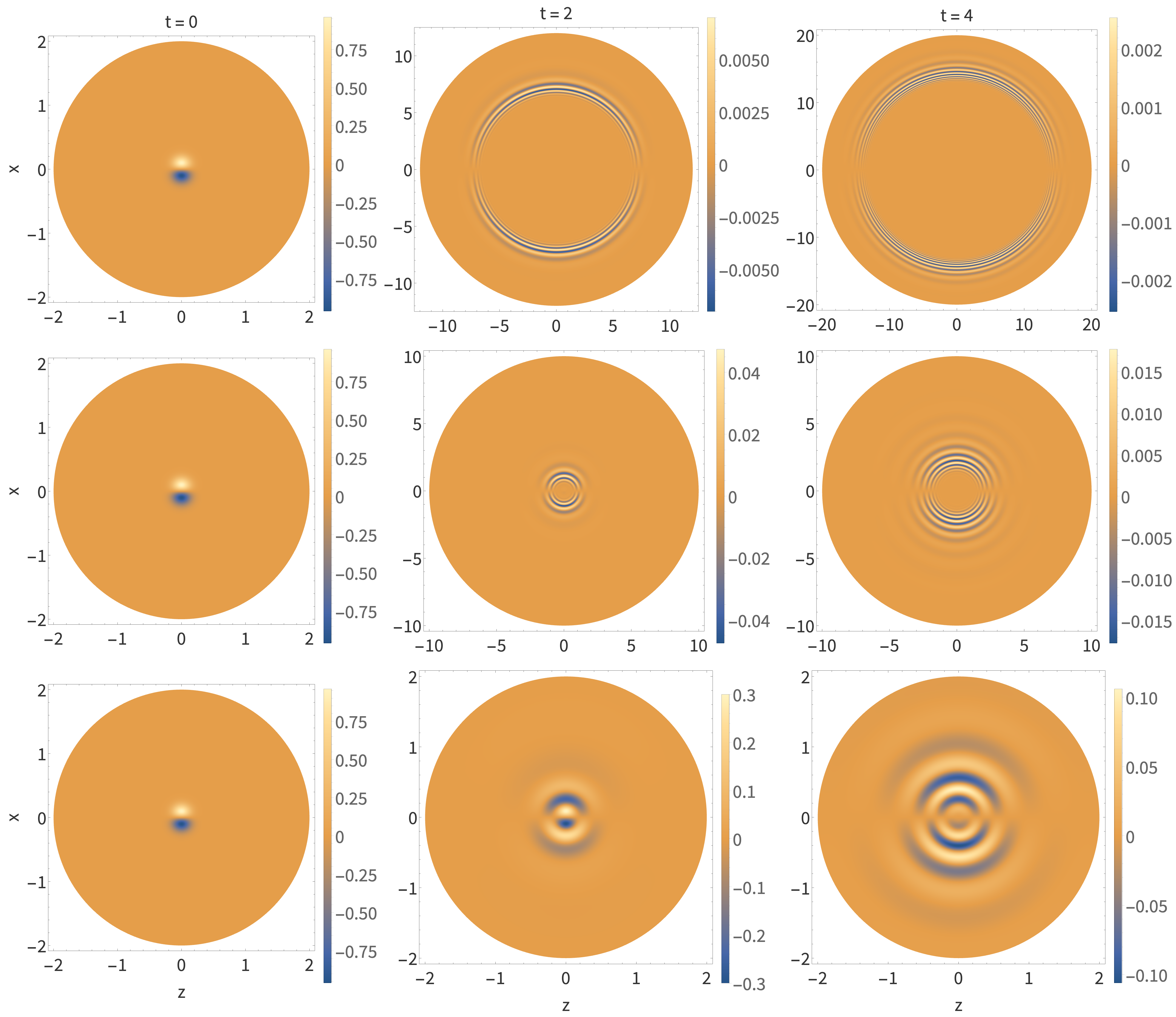}
\caption{Anisotropic 3D solution with $\kappa=\rho=\delta=1$ and $\sigma=1/10$ for $t=0,2,4$ (from right to left panels) and for $\alpha=9/10,1/2,1/10$ (from top to bottom panels).}
\label{fig:N3_s110_An}
\end{figure}
The machinery developed so far can be applied to investigate the propagation of anisotropic initial conditions even though the peridynamic kernel is isotropic. Indeed, according to the hypothesis of Theorem \ref{linsol3D}, an anisotropic 3D Cauchy problem is given by two initial conditions which are functions\footnote{Here $\n$ indicates the direction subtended by the angles $\left(\theta,\phi\right)$. To lighten the notation, in the remaining part of the work, we will use $\n$ to indicate }
\begin{align}
v_0=v_0(r,\n)\qquad\text{and}\qquad
v_1=v_1(r,\n)\,,
\label{eq:51}
\end{align}
where, in polar coordinates, $v_0,v_1\in\mathcal{S}(\R)\times L^2\left( \left[ 0,\pi \right]\times\left[ 0,2\pi \right] \right)$. Therefore, they can be written in terms of the spherical harmonic $Y_{\ell m}(\n)$ basis, namely
\begin{align}
v_0(r,\n)=\sum_{\ell=0}^\infty\sum_{m=-\ell}^\ell v^0_{\ell m}(r) Y_{\ell m}(\n)\,,\nonumber\\
v_1(r,\n)=\sum_{\ell=0}^\infty\sum_{m=-\ell}^\ell v^1_{\ell m}(r) Y_{\ell m}(\n)\,,
\label{eq:52}
\end{align}
where
\begin{align}
v^0_{\ell m}(r):=\int_{B^2(1)} v_0(r,\n)Y^*_{\ell m}(\n)d\n\,,\nonumber\\
v^1_{\ell m}(r):=\int_{B^2(1)} v_1(r,\n)Y^*_{\ell m}(\n)d\n\,.
\label{eq:multipoles}
\end{align}
Here $v^0_{\ell m},v^1_{\ell m}(r):\R^+\rightarrow\mathbb{C}$, $v^0_{\ell m},v^1_{\ell m}(r)\in\mathcal{S}(\R)$ and $B_1(0)$ is the unitary 2D sphere. In this case, we have  the explicit solution \eqref{eq:sol3D}  given by
\begin{align}
u\left( t,r,\n \right)=&\frac{2}{\pi}
\sum_{L=0}^\infty\sum_{M=-L}^L\,Y_{LM}\left( \n \right)
\int_0^\infty \xi^2j_L\left( \xi r \right)\,\cos\left(\omega(\xi)\,t\right)
\int_0^\infty v^0_{LM}(s)s^{2}\,j_L\left( \xi s \right)ds
d\xi\nonumber\\
&+\frac{2}{\pi}
\sum_{L=0}^\infty\sum_{M=-L}^L\,Y_{LM}\left( \n \right)
\int_0^\infty \xi^2j_L\left( \xi r \right)\,\frac{\sin\left(\omega(\xi)\,t\right)}{\omega(\xi)}
\int_0^\infty v^1_{LM}(s)s^{2}\,j_L\left( \xi s \right)ds
d\xi\,.
\label{eq:general_factorization}
\end{align}
The detailed proof of \eqref{eq:general_factorization} for a generic angular dependence is reported in \ref{app:A}. The remarkable technical point is that the factorisation occurring between the radial and the angular dependence in \eqref{eq:general_factorization} is general. Moreover, the numerical advantage of this chosen form stands in the fact that the needed numerical scheme to study the anisotropic initial conditions is the same as the isotropic one, namely a countable series two-dimensional integrals. This property is a general feature of the above-mentioned factorisation between the angular and the radial dependences and follows from the explicit expansion in the spherical harmonics basis.

As an informative science case study, we consider the 3D solution \eqref{eq:general_factorization} with initial conditions given by
\begin{align}
v_0(r,\theta)=&\left(2\pi\right)^{3/2}\exp\left( -\frac{r^2}{2\,\sigma^2} \right)r\cos\theta\nonumber\\
=&4\pi^2\sqrt{\frac{2}{3}}\exp\left( -\frac{r^2}{2\,\sigma^2} \right)r\,Y_{10}\left( \n \right)\,,\nonumber\\
v_1(r,\theta)=&0\,.
\label{eq:An_In_Con}
\end{align}
To our ends, it is then enough to evaluate the multipoles of the initial conditions in Eqs.~\eqref{eq:multipoles} to make use of the factorised solution \eqref{eq:general_factorization}. We straightforwardly have that
\begin{align*}
v^0_{\ell m}(r)=&4\pi^2\sqrt{\frac{2}{3}}\exp\left( -\frac{r^2}{2\,\sigma^2} \right)r
\int_{B^2(1)}Y_{10}\left( \n \right)Y^*_{\ell m}(\n)d\n\\
=&4\pi^2\sqrt{\frac{2}{3}}\exp\left( -\frac{r^2}{2\,\sigma^2} \right)r\delta_{\ell 1}\,\delta_{m0}\,,\\
v^1_{\ell m}(r)=&0\,.
\end{align*}
The Kronecker delta for $v^0_{\ell m}$ follows from the orthonormality of the spherical harmonics. From a physical viewpoint, the fact that only $\ell=1$ multipoles survive means that the we are dealing with a dipolar anisotropy, namely a our initial conditions have a front-back asymmetry. The numerical solutions for this dipolar anisotropy and different values of $\alpha$ are shown in Fig.~\ref{fig:N3_s110_An}.

Here we have considered just the case $\sigma/\delta=10^{-1}\ll 1$ where the dispersive behavior is more evident. It is evident from the numerical solutions that the ``shape" of the anisotropy is preserved as time passes by. This is a consequence of the factorization already discussed. All the other features regarding the dependence of the speed of propagation and the attenuation of the overall amplitude are manifest in accordance with what already pointed out in the radial case.

\section*{Conclusions and future works} 

The analysis of the dispersive relation for the peridynamics evolution of a two- and three-dimensional system has been deepened in this work through suitable numerical and analytical findings. We show that cases with different dimensionalities exhibit the dependence on the nonlocal interaction length $\delta$ in their dispersive relation, thus enlightening the universal characteristic of this dependence. Moreover, the dispersive relations are found to depend solely on the modulus of the Fourier variable $\xi$ and the two scenario related to small ($\xi\delta\ll 1$) and large scales ($\xi\delta\gg 1$) are analyzed in details.

\section*{Acknowledgments} AC is a member of Gruppo Nazionale per il Calcolo Scientifico (GNCS), GF of Gruppo Nazionale per la Fisica Matematica (GNFM), GMC and FM of Gruppo Nazionale per l'Analisi Matematica, la Probabilit\`{a} e le loro Applicazioni (GNAMPA) of the Istituto Nazionale di Alta Matematica (INdAM).
This work was partially supported by the project ``Research for Innovation'' (REFIN) - POR Puglia FESR FSE 2014-2020 - Asse X - Azione 10.4 (Grant No. CUP - D94120001410008); Fundacao para a Ciencia e a Tecnologia (FCT) under the program ``Stimulus'' with grant no. CEECIND/04399/2017/CP1387/CT0026 and through the research project with ref. number PTDC/FIS-AST/0054/2021; Italian Ministry of Education, University and Research under the Programme Department of Excellence Legge 232/2016 (Grant No. CUP - D94I18000260001). GMC expresses its gratitude to  
HIAS - Hamburg Institute for Advanced Study for their warm hospitality.

\bibliographystyle{abbrvnat}
\bibliography{biblio.bib}

\newpage
\appendix
\section{Proof of Eq.~\eqref{eq:general_factorization}}
\label{app:A}
To prove that the initial conditions \eqref{eq:51} with the multipolar decompositions \eqref{eq:52} actually admit Eqs.~\eqref{eq:general_factorization} for the proposed anisotropic Cauchy problem, let us first evaluate the Fourier transform of $v_0(r,\n)$. Without any loss of generality, we write a generic anisotropy in the form
\begin{equation}
v_0(r,\n)=\sum_{L=0}^\infty\sum_{M=-L}^L v^0_{L M}(r) Y_{L M}(\n)\,,
\end{equation}
where $\n:=\left( \sin\theta\cos\phi,\sin\theta\sin\phi,\cos\phi \right)$ is the direction subtended by the position $\x$ and $Y_{LM}$ are the spherical harmonics. The advantage of using $Y_{LM}$ stands in the fact that they provide a basis for all space of the functions $f\left(\n\right)$, i.e. we can express any anisotropy in the 3D space as a linear combination of spherical harmonics. Moreover, the dependences on $\theta$ and $\phi$ are respectively discretised by the indices $\ell$ and $m$.

Hence, according to our conventions for the Fourier transform, we have
\begin{align}
\widehat{v_0}(\xi,\n_\xi)
=&\frac{1}{\left( 2\pi \right)^3}
\int_0^\infty\int_0^\pi\int_0^{2\pi}v_0(r,\n)\,e^{i\,r\xi\,\n\cdot\n_\xi}\,r^2\sin\theta\,dr d\theta d\phi
\nonumber\\
=&\frac{1}{\left( 2\pi \right)^3}\sum_{L=0}^\infty\sum_{M=-L}^L
\int_0^\infty\int_0^\pi\int_0^{2\pi}v^0_{LM}(r)Y_{LM}\left( \n \right)\,e^{i\,r\xi\,\n\cdot\n_\xi}\,r^2\sin\theta\,dr d\theta d\phi\,,
\label{eq:A1}
\end{align}
where $\n$ and $\n_\xi$ are respectively the direction on the unitary sphere of $\x$ and $\Xi$. To our ends, we use the spherical harmonics decomposition for the exponential
\begin{equation}
e^{i\,r\xi\,\n\cdot\n_\xi}
=4\pi\sum_{\ell=0}^\infty\sum_{m=-\ell}^\ell i^\ell\,j_\ell\left( \xi r \right)\,Y_{\ell m}\left( \n \right)Y^*_{\ell m}\left( \n_\xi \right)\,,
\label{eq:A2}
\end{equation}
where $j_\ell(x)$ are the spherical Bessel functions of $\ell$-th order and $*$ denotes the complex conjugation. In this way, we can write Eq.~\eqref{eq:A1} as
\begin{align}
\widehat{v_0}(\xi,\n_\xi)
=&\left(-1\right)^M\frac{4\pi}{\left( 2\pi \right)^3}
\sum_{\ell=0}^\infty\sum_{m=-\ell}^\ell
\sum_{L=0}^\infty\sum_{M=-L}^L i^\ell\,Y^*_{\ell m}\left( \n_\xi \right)
\int_0^\infty v^0_{LM}(r)r^{2}\,j_\ell\left( \xi r \right)dr
\nonumber\\
&\times\underbrace{\int_0^\pi\int_0^{2\pi}Y^*_{L-M}\left( \n \right)\,Y_{\ell m}\left( \n \right)\,\sin\theta d\theta d\phi}_{=\delta_{\ell L}\delta_{m -M}}
\nonumber\\
=&\sum_{L=0}^\infty\sum_{M=-L}^L\left(-1\right)^M\frac{4\pi}{\left( 2\pi \right)^3}
 i^L\,Y^*_{L -M}\left( \n_\xi \right)
\int_0^\infty v^0_{LM}(r)r^{2}\,j_L\left( \xi r \right)dr\,,
\label{eq:A3}
\end{align}
where
in the second line we have used the orthonormal condition of the spherical harmonics
\begin{equation*}
\int_0^\pi\int_0^{2\pi}Y_{\ell m}\left( \n \right)Y^*_{\ell' m'}\left( \n \right)\,\sin\theta d\theta d\phi=\delta_{\ell \ell'}\delta_{mm'}\,.
\end{equation*}
In the same way, we also have that
\begin{align}
\widehat{v_1}(\xi,\n_\xi)
=&\sum_{L=0}^\infty\sum_{M=-L}^L\left(-1\right)^M\frac{4\pi}{\left( 2\pi \right)^3}
 i^L\,Y^*_{L -M}\left( \n_\xi \right)
\int_0^\infty v^1_{LM}(r)r^{2}\,j_L\left( \xi r \right)dr\,,
\label{eq:A4}
\end{align}

Now we can finally prove our solution directly from the theorem \ref{linsol3D} with the aim of Eqs.~\eqref{eq:A3}. We get
\begin{align}
u(t,r,\n)
=&\int_0^\infty\int_0^\pi\int_0^{2\pi} e^{-i\,\xi r\,\n\cdot \n_\xi}\widehat{v_0}(\xi,\n_\xi)\cos\left(\omega(\xi)\,t\right)
\xi^2\sin\theta_\xi d\xi d\theta_\xi d\phi_\xi
\nonumber\\
&+\int_0^\infty\int_0^\pi\int_0^{2\pi} e^{-i\,\xi r\,\n\cdot \n_\xi}\widehat{v_1}(\xi,\n_\xi)\frac{\sin\left(\omega(\xi)\,t\right)}{\omega(\xi)}
\xi^2\sin\theta_\xi d\xi d\theta_\xi d\phi_\xi\nonumber\\
=&\frac{\left(4\pi\right)^2}{\left(2\pi \right)^3}\sum_{\ell=0}^\infty\sum_{m=-\ell}^\ell
\sum_{L=0}^\infty\sum_{M=-L}^L \left(-i\right)^\ell\,i^L\,Y_{\ell m}\left( \n \right)
\underbrace{\int_0^\pi\int_0^{2\pi}\,
 \left(-1\right)^MY^*_{L-M}\left( \n_\xi \right)Y^*_{\ell m}\left( \n_\xi \right)\sin\theta_\xi d\theta_\xi d\phi_\xi}_{=\delta_{L\ell}\delta_{Mm}}\nonumber\\
&\times\int_0^\infty \xi^2j_\ell\left( \xi r \right)\,\cos\left(\omega(\xi)\,t\right)
\int_0^\infty v^0_{LM}(s)s^{2}\,j_L\left( \xi s \right)ds
d\xi\nonumber\\
&+\frac{\left(4\pi\right)^2}{\left(2\pi \right)^3}\sum_{\ell=0}^\infty\sum_{m=-\ell}^\ell
\sum_{L=0}^\infty\sum_{M=-L}^L \left(-i\right)^\ell\,i^L\,Y_{\ell m}\left( \n \right)
\underbrace{\int_0^\pi\int_0^{2\pi}\,
 \left(-1\right)^MY^*_{L-M}\left( \n_\xi \right)Y^*_{\ell m}\left( \n_\xi \right)\sin\theta_\xi d\theta_\xi d\phi_\xi}_{=\delta_{L\ell}\delta_{Mm}}\nonumber\\
&\times\int_0^\infty \xi^2j_\ell\left( \xi r \right)\,\frac{\sin\left(\omega(\xi)\,t\right)}{\omega(\xi)}
\int_0^\infty v^1_{LM}(s)s^{2}\,j_L\left( \xi s \right)ds
d\xi\nonumber\\
=&\frac{2}{\pi}
\sum_{L=0}^\infty\sum_{M=-L}^L\,Y_{LM}\left( \n \right)
\int_0^\infty \xi^2j_L\left( \xi r \right)\,\cos\left(\omega(\xi)\,t\right)
\int_0^\infty v^0_{LM}(s)s^{2}\,j_L\left( \xi s \right)ds
d\xi\nonumber\\
&+\frac{2}{\pi}
\sum_{L=0}^\infty\sum_{M=-L}^L\,Y_{LM}\left( \n \right)
\int_0^\infty \xi^2j_L\left( \xi r \right)\,\frac{\sin\left(\omega(\xi)\,t\right)}{\omega(\xi)}
\int_0^\infty v^1_{LM}(s)s^{2}\,j_L\left( \xi s \right)ds
d\xi\,,
\end{align}
where we have made use again of the decomposition \eqref{eq:A2} and of the properties of the spherical harmonics. 

\end{document}